\documentclass[11pt,a4paper]{article}
\usepackage{bbm,times,url,graphicx,anysize,amssymb,amsmath,amsthm,multicol}
\title{Combined degree and connectivity conditions for $H$-linked graphs}
\author{\textsc{Florian Pfender}\\
\normalsize Universit{\"a}t Rostock\\
\normalsize Institut f{\"u}r Mathematik\\
\normalsize D-18055 Rostock, Germany\\
\normalsize \texttt{Florian.Pfender@uni-rostock.de}
\date{}}

\newcommand{\N}{{\mathbbm N}}
\newcommand{\cS}{{\mathcal S}}
\newcommand{\cA}{{\mathcal A}}
\newcommand{\cB}{{\mathcal B}}

\newtheorem{theorem}{Theorem}[section]
\newtheorem{lemma}[theorem]{Lemma}
\newtheorem{fact}[theorem]{Fact}
\newtheorem{definition}[theorem]{Definition}
\newtheorem{corollary}[theorem]{Corollary}

\newtheorem{claim}{Claim}
\newtheorem{case}{Case}[theorem]

\begin{document}
\maketitle

\begin{abstract}
For a given multigraph $H$, a graph $G$ is $H$-linked, if $|G|\ge |H|$ and for every
injective map $\tau: V(H)\to V(G)$,
we can find internally disjoint paths in $G$, such that every
edge from $uv$ in $H$ corresponds to a $\tau(u)-\tau(v)$ path.

To guarantee that a $G$ is $H$-linked, you need a minimum degree larger than $\frac{|G|}{2}$. This
situation changes, if you know that $G$ has a certain connectivity $k$. Depending on $k$,
even a minimum degree independent of $|G|$ may suffice. Let $\delta(k,H,N)$ be the minimum number,
such that every $k$-connected graph $G$ with $|G|=N$ and $\delta(G)\ge \delta(k,H,N)$
is $H$-linked. We study bounds for this quantity. In particular, we find bounds for all multigraphs $H$ with
at most three edges, which are optimal up to small additive or multiplicative constants.
\end{abstract}

\section{Introduction and notation}
All graphs and multigraphs considered here are loopless. For concepts and notation not defined here we refer the reader 
to Diestel's book~(\cite{Di}). 

A {\em separation} of a graph $G$ consists of two sets $A,B\subseteq V(G)$ with
$A\cup B=V$ and no edges between $A\setminus B$ and $B\setminus A$. If $|A\cap B|=k$ then the separation is called a 
{\em $k$-separation}.

Now let $H$ be a multigraph. A graph $G$ is {\em $H$-linked}, if $|G|\ge |H|$ and for every
injective map $\tau: V(H)\to V(G)$,
we can find internally disjoint paths in $G$, such that every
edge from $uv$ in $H$ corresponds to a $\tau(u)-\tau(v)$ path.
This concept generalizes several concepts of connectivity studied before.
If $H$ is a star with $k$ edges (or a $k$-multi-edge), then $H$-linked graphs are exactly the
$k$-connected graphs.
If $H$ is a cycle  with $k$ edges, then $H$-linked graphs are exactly the
$k$-ordered graphs. Finally,
if $H$ is a matching with $k$ edges, then $H$-linked graphs are exactly the
$k$-linked graphs.

The following are easy facts about $H$-linked graphs. Detailed proofs for Facts~\ref{f1} and~\ref{f2}
can be found in~\cite{LWY}.
\begin{fact}\label{f1}
Let $H_1$ and $H_2$ be multigraphs and suppose that $H_2$ is a
submultigraph of $H_1$. Then every $H_1$-linked graph is
$H_2$-linked.
\end{fact}
\begin{fact}\label{f2}
Let $H_1$ and $H_2$ be multigraphs and suppose that one gets $H_2$
from $H_1$ through the identification of two non-adjacent 
vertices, one of
which has degree $1$. Then every $H_1$-linked graph is
$H_2$-linked.
\end{fact}

\begin{corollary}\label{c1}
Let $H$ be a multigraph without isolated vertices. Then every
$|E(H)|$-linked graph is $H$-linked.
\end{corollary}
\begin{fact}\label{f3}
Let $H$ be a multigraph with a $k$-multi-edge. Then, every $H$-linked graph is
$(|H|-2+k)$-connected.
\end{fact}
The minimum degree required for a graph to be $k$-linked is well understood. Kawabarayashi, Kostochka and Yu
prove the following sharp bounds.
\begin{lemma}[\cite{KKY}]\label{klinked4}
Let $G$ be a graph on $N\ge 2k$ vertices with minimum degree
$$\delta(G)\ge 
\begin{cases}
\frac{N+2k-3}{2}, & \mbox{ if } ~N\ge 4k-1\\
\frac{N+5k-5}{3}, & \mbox{ if } ~3k\le N\le 4k-2\\
N-1, & \mbox{ if } ~2k\le N\le 3k-1
\end{cases} .
$$
Then $G$ is $k$-linked.
\end{lemma}
Note that the degree bounds above imply that the graph $G$ is $(2k-1)$-connected. Further, if $G$ has a $(2k-1)$-separation $(A,B)$,
the bounds allow missing edges in $G[A]$ and $G[B]$ only inside $A\cap B$. 
On the other hand, if $G$ is $2k$-connected, then an average (and thus a minimum) degree constant in $N$ is sufficient,
the best known bound was found by Thomas and Wollan.
\begin{theorem}[\cite{TW}]\label{klinked}
If $G$ is $2k$-connected and $G$ has at least $5k|V(G)|$ edges,
then $G$ is $k$-linked.
\end{theorem}

For $k=3$, Thomas and Wollan strengthen this bound to a sharp bound. Given a graph $G$ and a set $X\subset V(G)$,
the pair $(G,X)$ is called {\em linked}, if for every set $\{ x_1,\ldots,x_k,y_1,\ldots,y_k\}\subseteq X$ of
$2k\le |X|$ disjoint vertices, there are $k$ disjoint $x_i-y_i$ paths with no internal vertices in $X$.
\begin{theorem}[\cite{TW2}]\label{3linked}
Let $G$ be a graph, an let $X\subset V(G)$ with $|X|=6$.
If $G$ has no $5$-separation $(A,B)$ with $X\subseteq A$ and $G$ has at least $5|V(G)|-26$ edges outside of $G[X]$,
then $(G,X)$ is linked.
\end{theorem}
\begin{corollary}[\cite{TW2}]
If $G$ is $6$-connected and $G$ has at least $5|V(G)|-14$ edges,
then $G$ is $3$-linked.
\end{corollary}
\begin{corollary}[\cite{TW2}]
If $G$ is $6$-connected and $\delta(G)\ge 10$,
then $G$ is $3$-linked.
\end{corollary}

Similarly, bounds have been known for a long time for the case $k=2$.
\begin{theorem}[\cite{J}]
Let $G$ be a $4$-connected graph, which is either non-planar or triangulated. Then $G$ is $2$-linked.
\end{theorem}
\begin{corollary}
If $G$ is $4$-connected and $\| G\|\ge 3|G|-6$, then $G$ is $2$-linked.
\end{corollary}
\begin{corollary}
If $G$ is $4$-connected and $\delta(G)\ge 6$, then $G$ is $2$-linked.
\end{corollary}

In a sense, there is a rather sharp threshold for $k$-linked graphs. If $\kappa(G)=2k-2$, then $G$ is not $k$-linked. 
If $\kappa(G)=2k-1$, we need to give very strong (linear) degree conditions to guarantee that $G$ is $k$-linked. If  
$\kappa(G)=2k$, then weak (constant) degree conditions suffice.
Our program is to study similar behavior in the more general setting of $H$-linked graphs. 
In particular, we want to study the following quantity.

\begin{definition}
Let $H$ be a multigraph, and let $k\ge 0$. Choose $N\ge k+1$ large enough, so that $K^{N}$ is $H$-linked, and define
$$
\delta(k,H,N):=\min\{\delta\in \N_{\ge k}:\mbox{ every $k$-connected graph on $N$ vertices with $\delta(G)\ge \delta$ 
is $H$-linked}\}.
$$
\end{definition}
Due to the following simple fact, we will restrict our attention to multigraphs $H$ without isolated vertices
for the rest of the paper.
\begin{fact}
Let $H$ be a multigraph. Then $\delta(k+1,H\cup v,N+1)=\delta(k,H,N)+1$.
\end{fact}

We can state some of the above Theorems and facts along the lines of our program.
\begin{theorem}
Let $H$ be a connected bipartite multigraph with $\ell$ edges, where one of the two parts of the bipartition contains only one vertex.
Then 
$$
\delta(k,H,N)=\begin{cases}

\left\lceil\frac{N+\ell-2}{2}\right\rceil,&\mbox{ for }k< \ell,\\
k, &\mbox{ for }k\ge \ell.
\end{cases}
$$
\end{theorem}
\begin{theorem}
$$
\begin{array}{rcccll}\label{Tlk2}
 &&\delta(k, \ell~K^2,N)
&=&\left\lceil\frac{N+2\ell-3}{2}\right\rceil+o(1), & \mbox{ if }k<2\ell,\\
\max\{ 2\ell+2-o(1),k\} &\le&\delta(k, \ell~K^2,N)&\le& \max\{ 10\ell,k\}, & \mbox{ if }k\ge 2\ell.
\end{array}
$$
\end{theorem}
\begin{theorem}\label{T3k2}
$$
\begin{array}{rcccll}
&& \delta(k, 3~K^2,N) 
&=&\left\lceil\frac{N+3}{2}\right\rceil+o(1), & \mbox{ if }k<6,\\
\max\{ 8-o(1),k\} &\le&\delta(k, 3~K^2,N)&\le& \max\{10,k\}, & \mbox{ if }k\ge 6.
\end{array}
$$
\end{theorem}
\begin{theorem}\label{T2k2}
$$
\begin{array}{rcccll}
&& \delta(k, 2~K^2,N) 
&=&\left\lceil\frac{N+1}{2}\right\rceil+o(1), & \mbox{ if }k<4,\\
\max\{ 6-o(1),k\} &\le&\delta(k, 2~K^2,N)&\le& \max\{6,k\}, & \mbox{ if }k\ge 4.
\end{array}
$$
\end{theorem}

To get the lower bounds in Theorems~\ref{Tlk2},~\ref{T3k2} and~\ref{T2k2}, construct a not $\ell$-linked
but $(2\ell+1)$-connected graph  from a planar, not triangulated $5$-connected graph by adding $2\ell-4$
{\em universal} vertices, which are connected to all other vertices of the graph.
We will finish this section with a small new result.
\begin{theorem}
$$
\delta(k, K^3,N)= \begin{cases}
\left\lceil\frac{N}{2}\right\rceil, & \mbox{ if }k< 2,\\
\left\lceil\frac{N+2}{3}\right\rceil, & \mbox{ if }k= 2,\\
k, & \mbox{ if }k\ge 3.
\end{cases}
$$
\end{theorem}
\begin{proof}
Let $\{ x,y,z\}\subseteq V(G)$.
For $k<2$, the statement is trivial, as only $2$-connected graphs can be $K^3$-linked, and we need $\delta(G)\ge \frac{N}{2}$
to guarantee $\kappa(G)\ge 2$, in which case the lower bound for $k=2$ gives the result. For $k=2$, 
let $C$ be a longest cycle in $G$ containing $\{ x,y\}$. 
Then, with a standard Dirac type argument (\cite{D}),
$|C|\ge 2\delta(G)\ge \frac{2N+4}{3}$. If $z$ is on $C$, we are done. Otherwise,
$|N(z)\cap C|\ge 3$, and $z$ has at least two neighbors on at least one of $xCy$ and 
$yCx$, so we can find a cycle through $x$, $y$ and $z$. On the other hand, the graph consisting of
three complete graphs on $\frac{N-2}{3}$ 
(rounded up or down appropriately) vertices, each of them completely connected to two independent vertices, shows the sharpness
of the bound.

For $k\ge 3$, the statement is very easy again, as every $3$-connected graph admits a cycle through any three
given vertices, and is thus $K^3$-linked.
\end{proof}

Note that the only multigraphs with three edges 
we have not considered yet are $P^4$, $K^2\cup P^3$ and $K^2\cup C^2$, 
where $C^2$ denotes two vertices connected by a double edge. We will find good
bounds for these graphs in the following three sections.

\section{$H=P^4$}\label{sp4}

Let us start this section with a definition.
\begin{definition}
Let $G$ be a graph, and let $\{ a,a',b,b',c,c'\} \subseteq V(G)$. Then $(G,\{ b,b'\} ,\{ c,c'\}, (a,a'))$ is
an {\em obstruction} if for any three vertex disjoint paths from $\{ a,b,b'\}$ to $\{ a',c,c'\}$, one path is from $a$ to $a'$.
\end{definition}
Note that if $G$ is a graph which does not contain a path through $a,c,b,a'\in V(G)$ in this order, we can
construct an obstruction $(G_{b,c},\{ b,b'\} ,\{ c,c'\}, (a,a'))$ from $G$ through addition of two vertices $\{ b',c'\}$
with $N(b')=N(b)$ and $N(c')=N(c)$. Thus, if we want to find bounds on $\delta(k,P^4,n)$, it will be helpful to know
the structure of obstructions.

Yu has characterized obstructions in~\cite{Y}.
We will be concerned mostly with connectivity $k\ge 4$, 
so we will state his results here only for $4$-connected graphs. In particular, we omit case (4) in the following definition.
\begin{definition}[\cite{Y}]
Let $G$ be a graph, and $\{ a,b,b'\},\{ a',c,c'\} \subseteq V(G)$. Suppose  $\{ a,b,b'\}\ne\{ a',c,c'\}$,
and assume that $G$ has no proper $3$-separation $(G_1,G_2)$ such that $\{ a,b,b'\} \subseteq G_1$ and
$\{ a,c,c'\} \subseteq G_2$. Then we call
$(G,(a,b,b'),(a',c,c'))$ a {\em rung} if one of the following is satisfied, up to permutation of $\{ b,b'\}$
and $\{ c,c'\}$.

\begin{tabular}{l p{0.65\textwidth} p{0.15\textwidth}}
(1)&
$a=a'$ or $\{ b,b'\} = \{ c,c'\}$; & \vspace{0in}\scalebox{0.2}{\includegraphics{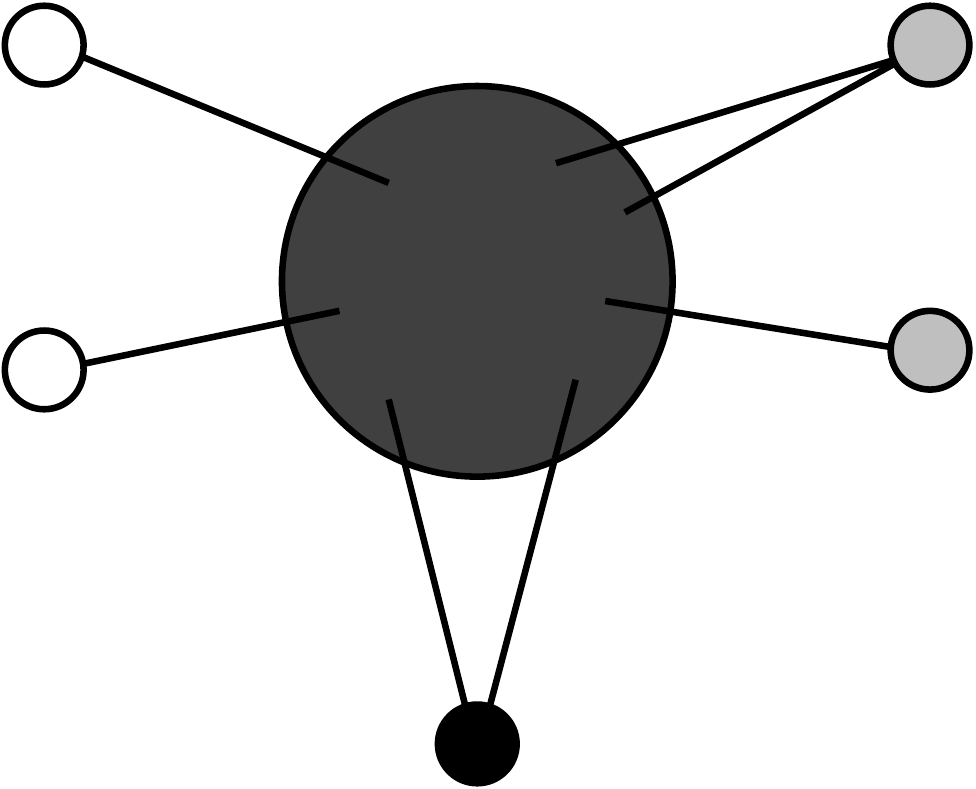}}\\ \hline

(2)&
$b=c$ and $(G-b, b',c',a',a)$ is plane; & \vspace{0in}\scalebox{0.2}{\includegraphics{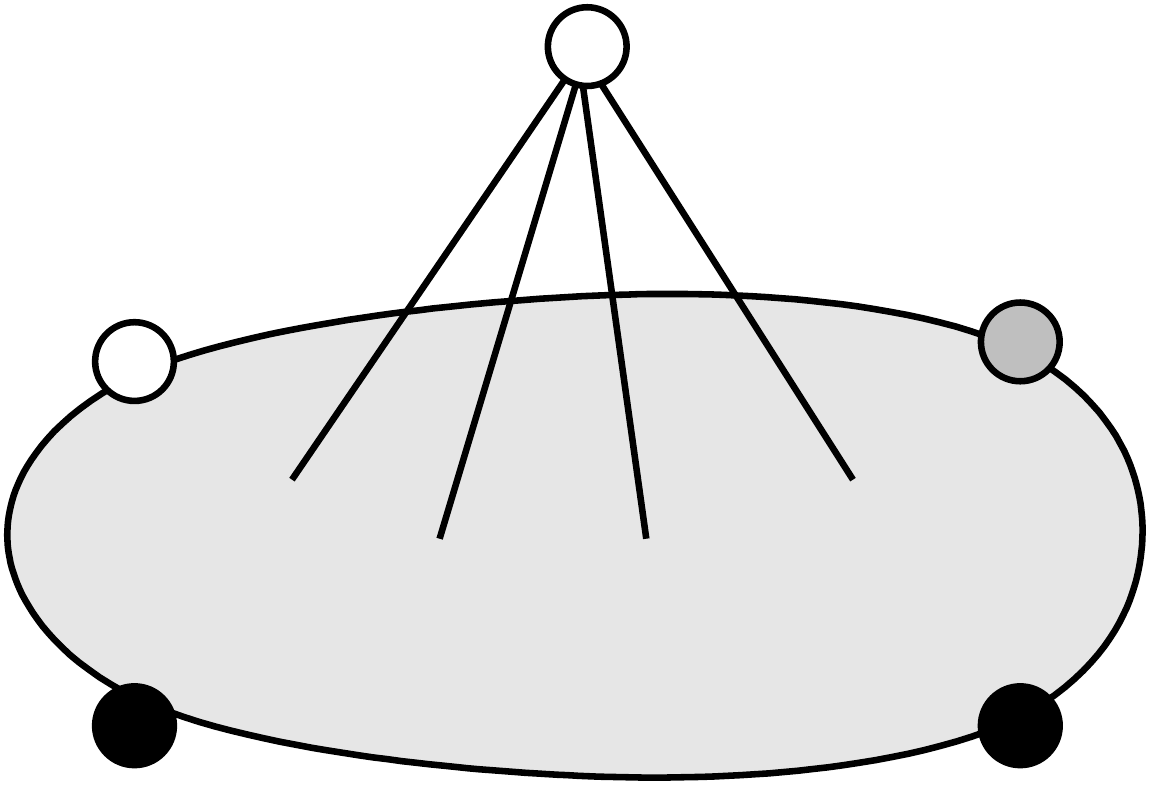}}\\

\hline

(3)&
$\{ a,b,b'\}\cap \{a',c,c'\} =\emptyset$ and $(G,b,a,b',c',a',c)$ is plane; 
& \vspace{0in}\scalebox{0.2}{\includegraphics{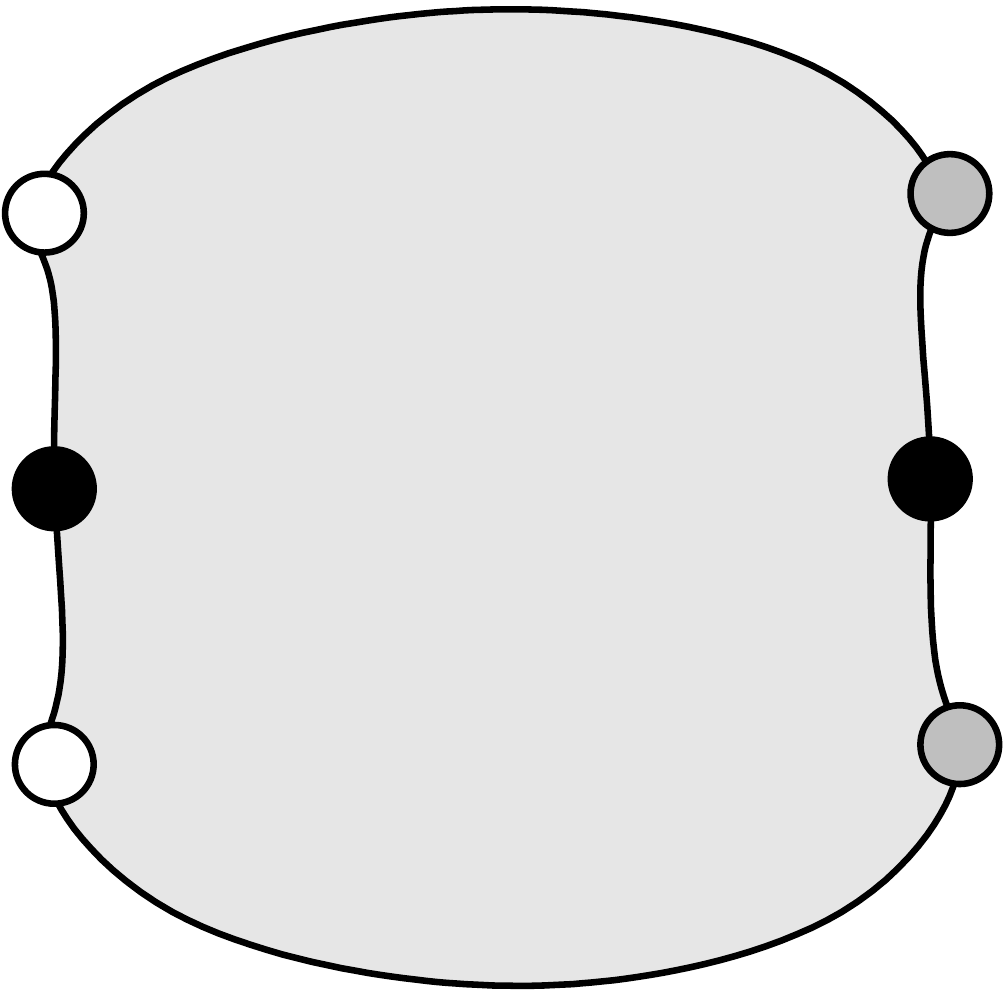}}\\
\hline

(5)&
$\{ a,b,b'\}\cap \{a',c,c'\} =\emptyset$, $(G,b,c,a',a)$ is plane, and $G$ has a separation $(G_1,G_2)$, such that 
$V(G_1\cap G_2)=\{ z,a\}$ (or $V(G_1\cap G_2)=\{ z,a'\}$), $\{ b,c,a,a'\} \subseteq G_1$, $\{b',c'\} \subseteq G_2$,
and $(G_2,b',c',z,a)$ (or $(G_2,b',c',a',z)$) is plane; & \vspace{0in}\scalebox{0.2}{\includegraphics{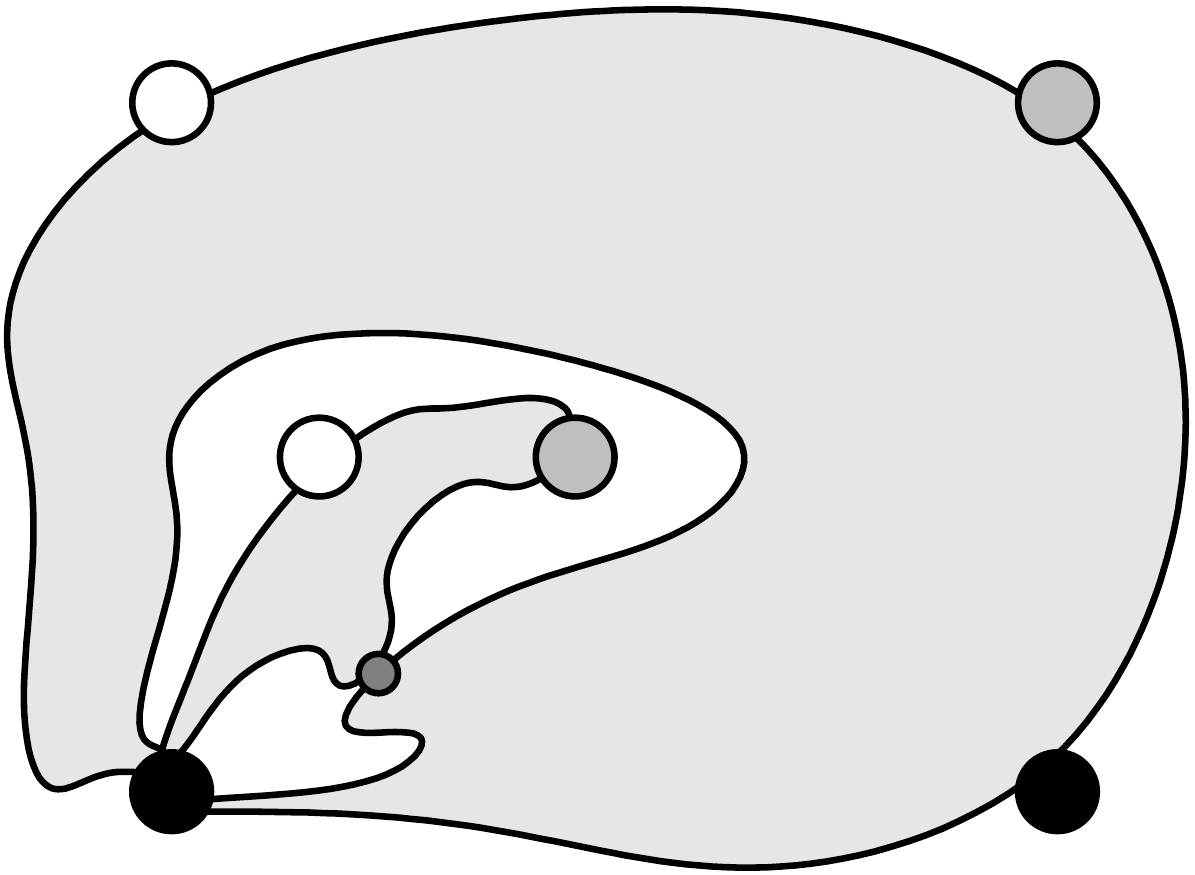}}\\

(6)&
$\{ a,b,b'\}\cap \{a',c,c'\} =\emptyset$, and there are pairwise edge disjoint subgraphs $G_1$, $G_2$ and $M$ of $G$
such that $G=G_1\cup G_2\cup M$, $V(G_1\cap M)= \{ u,w\}$, $V(G_2\cap M)= \{ p,q\}$, $G_1\cap G_2=\emptyset$, 
$\{ a,b,c\}\subseteq G_1$, $\{ a',b',c'\}\subseteq G_2$, and $(G_1,a,b,c,w,u)$ and $(G_2,a',c',b',p,q)$ are plane; 
& \vspace{0in}\scalebox{0.2}{\includegraphics{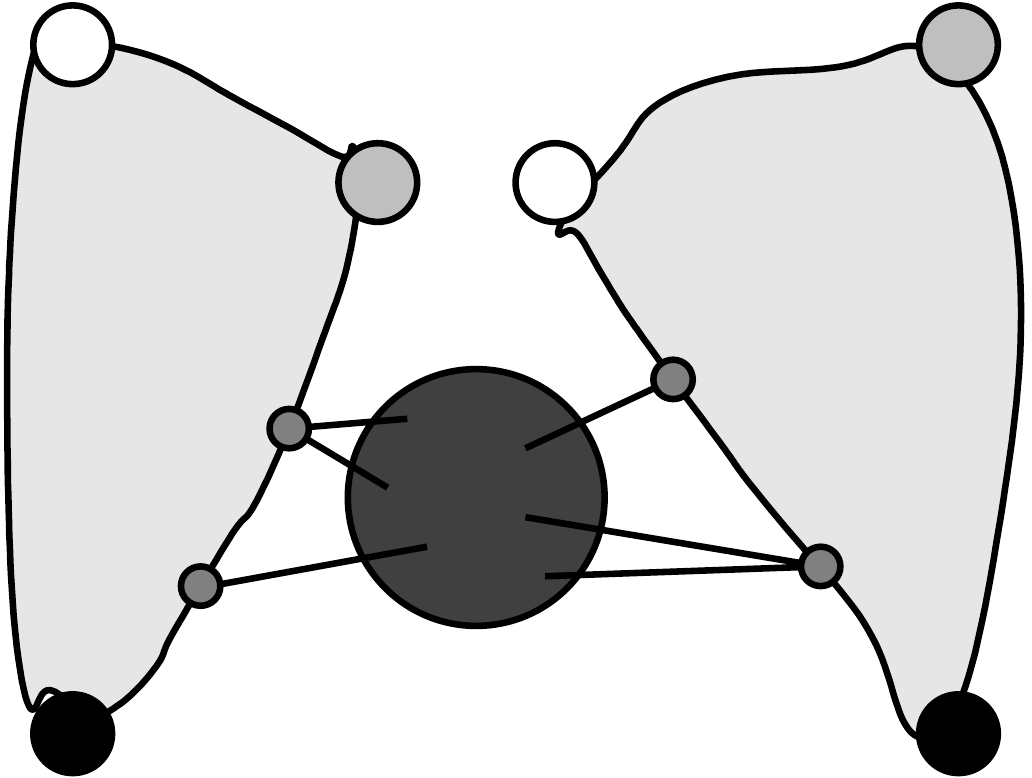}} \\
\hline

(7)&
$\{ a,b,b'\}\cap \{a',c,c'\} =\emptyset$, and there are pairwise edge disjoint subgraphs $G_1$, $G_2$ and $M$ of $G$
such that $G=G_1\cup G_2\cup M$,  $V(G_1\cap M)= \{ a,a',w\}$, $V(G_2\cap M)= \{ a,a',p\}$, $V(G_1\cap G_2)=\{ a,a'\}$, 
$\{ b,c\}\subseteq G_1$, $\{ b',c'\}\subseteq G_2$, and $(G_1,b,c,a',w,a)$ and $(G_2,c',b',a,p,a)$ are plane;
& \vspace{0in}\scalebox{0.2}{\includegraphics{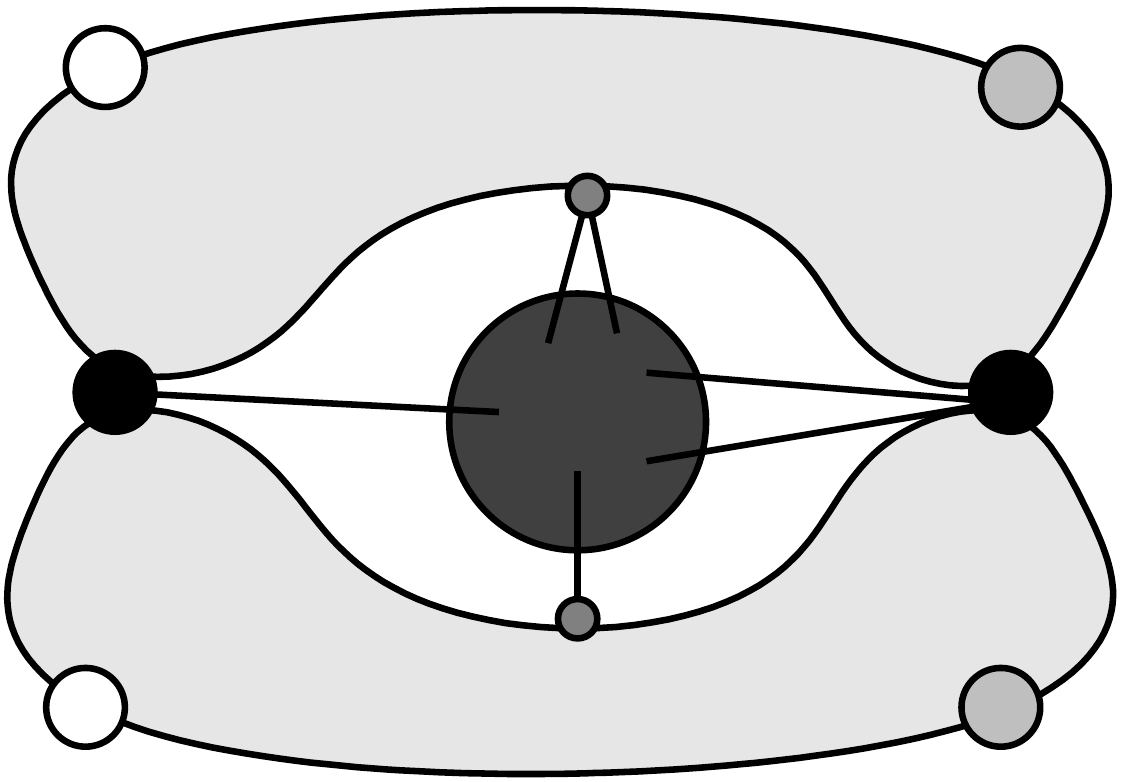}}\\
\end{tabular}

\end{definition}

\begin{definition}[\cite{Y}]
Let $L$ be a graph and let $R_1,\ldots ,R_m$ be edge disjoint subgraphs of $L$ such that
\begin{enumerate}
\item 
$(R_i,(v_{i-1},x_{i-1},y_{i-1}),(v_i,x_i,y_i))$ is a rung for $1\le i\le m$,
\item 
$V(R_i\cap R_j)=\{ v_i,x_i,y_i\} \cap \{ v_{j-1},x_{j-1},y_{j-1}\}$ for $1\le i<j\le m$,
\item
for $0\le i\le k\le j\le m$, we have $(v_i=v_j\implies v_i=v_k)$, $(x_i=x_j\implies x_i=x_k)$
and $(y_i=y_j\implies y_i=y_k)$,
\item
$L=(\bigcup_{i=1}^m R_i)+S$, where $S\subseteq \bigcup_{i=0}^m \{ v_ix_i,v_iy_i,x_iy_i\}$.
\end{enumerate}
Then we call $L$ a {\em ladder along $v_0\ldots v_m$}.
\end{definition}
Due to condition 4 in the last definition, we may assume that there are no edges in 
$R_i[v_{i-1},x_{i-1},y_{i-1}]$ and $R_i[v_i,x_i,y_i]$ for $1\le i\le m$.

For a sequence $S$, the {\em reduced sequence of $S$} is the sequence obtained from $S$ by removing a minimal number 
of elements such that consecutive elements differ. After these definitions, we are ready to formulate the 
version of the characterization theorem for obstructions, restricted to $4$-connected graphs.
\begin{theorem}[\cite{Y}]\label{char}
Let $G$ be a $4$-connected graph, $\{ a,b,b'\},\{ a',c,c'\}\subseteq V(G)$. Then the following are equivalent.
\begin{enumerate}
\item
$(G,\{ b,c\} ,\{ b',c'\} ,(a,a'))$ is an obstruction,
\item
$G$ has a separation $(J,L)$ such that $V(J\cap L)=\{ w_0,\ldots,w_n\}$, $(J,w_0,\ldots,w_n)$ is plane,
and $(L,(a,b,b'),(a',c,c'))$ is a ladder along $v_0\ldots v_m$, where $v_0=a$, $v_m=a'$ and $w_0\ldots w_n$
is the reduced sequence of $v_0\ldots v_m$.
\end{enumerate}
\end{theorem}
With the help of Theorem~\ref{char}, we can determine fairly sharp bounds for $\delta(k,P^4,N)$.

\begin{theorem}\label{thmp4}
$$
\begin{array}{rcccl}
&&\delta(k,P^4,N)&=&\left\lceil \frac{N+1}{2}\right\rceil ,\mbox{ if }k\le 3 \mbox{ and }N\ge 14,\\
\sqrt{N+1}&\le & \delta(4,P^4,N) &\le &\sqrt{N}+5,\\
\sqrt[3]{N}+2.7&\le & \delta(5,P^4,N) &\le &\sqrt[3]{N}+4.2+o(1)\le\sqrt[3]{N}+6,\\
6&\le & \delta(6,P^4,N) &\le &8,\\
&&\delta(6,P^4,N) &= &8,\mbox{ if }N\ge 418\\
&& \delta(k,P^4,N) &= &k,\mbox{ if }k\ge 7.
\end{array}
$$
\end{theorem}
\begin{proof}
\begin{case}
$k\le 3$
\end{case}
If $G$ has minimum degree $\delta(G)\ge \frac{N+1}{2}$ and a $3$-separation $(A,B)$, 
then $G[A]$ and $G[B]$ can have missing edges only inside 
$A\cap B$.  In this case, it is easy to check that $G$ is $P^4$-linked for $N\ge 6$. If $G$ has no $3$-separation, then $G$ is $4$-connected and the result follows from the case $k=4$ for $N\ge 14$.

To show that
$\delta(k,P^4,N)> \lfloor\frac{N}{2}\rfloor$ consider a graph $G$
consisting of two complete graphs $G_1$ and $G_2$ with $|G_1|=\lceil\frac{N+2}{2}\rceil$, $|G_2|=\lfloor\frac{N+2}{2}\rfloor$,
and $|G_1\cap G_2|=2$, and an additional edge $p_1p_4$ with $p_1\in V(G_1\setminus G_2)$ and $p_4\in V(G_2\setminus G_1)$.
If we choose $p_3\in V(G_1\setminus G_2)$ and $p_2\in V(G_2\setminus G_1)$, then $G$ contains no path passing through 
$p_1,p_2,p_3,p_4$ in order. 
\begin{case}\label{c2}
$k=4$
\end{case}
First, we will construct a graph $G$ demonstrating that $\delta(4,P^4,N)\ge \sqrt{N+1}$. 
Let $\delta\ge 4$. Let $Z_i$, $1\le i\le \delta -1$  be complete graphs with $|V(Z_i)|=\delta+1$. Let
$\{ a_i,b_i,x_i,y_i\}\subset V(Z_i)$, where $a_1=b_2$,
and otherwise the $V(Z_i)$ are disjoint. Let $V(G)=\{ p_1,p_4\}\cup \bigcup V(Z_i)$, and add all edges
$a_ib_{i+1},~x_ix_{i+1},~y_iy_{i+1}$ for $1\le i\le \delta-2$. Further, add the edge $p_1p_4$ and edges from
$p_1$ to the first $\delta-1$ vertices of the path $P=b_1b_2a_2b_3\ldots b_{\delta-1}a_{\delta-1}$, and from 
$p_4$ to the last $\delta-1$ vertices of $P$ (see Figure~\ref{p4_4}).
\begin{figure}[h]\label{p4_4}
\begin{center}
\scalebox{0.4}{\includegraphics{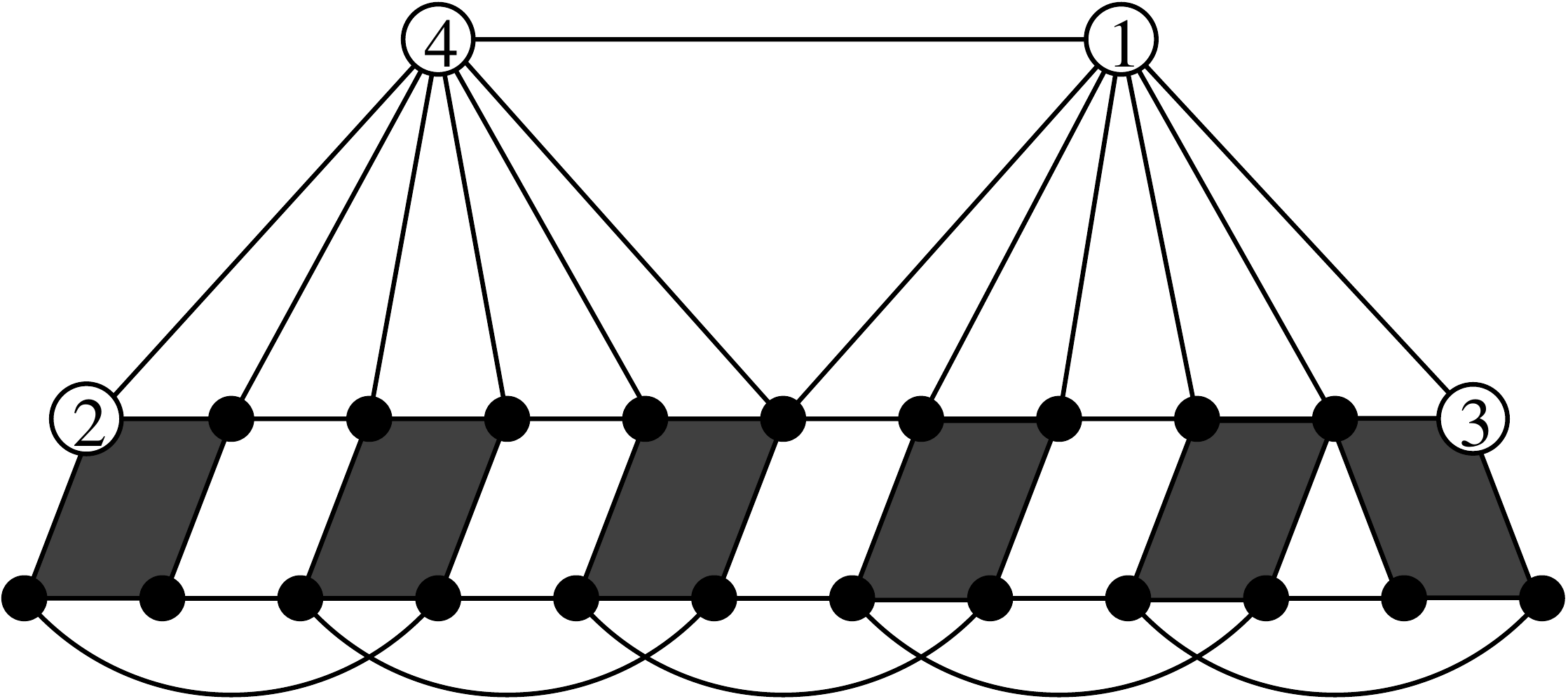}}
\end{center}
\caption{The graph $G$ in Case~\ref{c2}} 
\end{figure}
Then $\delta(G)=\delta$, $G$ is $4$-connected, and $N=|V(G)|=\delta^2$. Further, there is no path containing 
$p_1,~a_{\delta-1},~b_1,~p_4$ in order.

To show that $\delta(4,P^4,N)<\sqrt{N}+5$, assume that $G$ is a $4$-connected graph on $n$ vertices with minimum degree 
$\delta \ge 7$ (the statement is trivial for $\delta\le 6$),
and assume that $G$ contains vertices $a,~c,~b,~a'$, but no path contains the vertices in the given order.
Then $(G_{b,c},\{ b,b'\} ,\{ c,c'\} ,(a,a'))$ is an obstruction and has the structure described in Theorem~\ref{char}.
Let us focus on the structure of the first rung $R_1\subseteq L\subset G_{b,c}$. Since $N(b)=N(b')$, types
(3), (5), (5'), (6), (6'), and (7) are not possible. If $R_1$ is of type (2) or (2'), then it is in fact also of type (1)
by the same reasoning. 
Since $G$ is $4$-connected (and $G$ is obtained from $G_{b,c}$
by contracting $\{ b,b'\}$ and $\{ c,c'\}$), we can conclude that $R_1$ is of type (1) with 
$\{ b,b'\} =\{ x_1,y_1\}$ and $V(R_1)=\{ a,b,b',v_1\}$.

Similarly, $R_m$ is of type (1) with $V(R_m)=\{ a',c,c',v_{m-1}\}$. Thus, $a$ and $a'$ each have no neighbors outside 
of $J\cup \{ b,b',c,c'\}$ (and thus, each of them has at least $\delta-1$ neighbors in $J$). 
We may assume that $aa'\in E(J)$, otherwise we can add it. 
Due to Euler's formula, all triangles in $J$ are facial,  so 
\begin{equation*}
|V(J)|\ge |N(a)\cup N(a')| \ge 2(\delta-1)-|N(a)\cap N(a')|\ge 2\delta-3.
\end{equation*}
Therefore, again with Euler's formula, $J$ has a lot of outgoing edges.
\begin{equation}\label{e1}
|E(J,G\setminus J)|\ge \delta|V(J)|-(6|V(J)|-12).
\end{equation}
For every $1\le j\le m-2$, there are at most two vertices in $V(L)\cap N(v_j)\cap N(v_{j+1})\cap N(v_{j+2})$
due to the ladder structure of $L$. Every other vertex in $V(L)$ has at most two neighbors in $V(J)$.
Noting that $|V(J)|>n$ and $|V(L\cap J)|=n+1$, we have
$$
2|V(L)|\ge |E(J,G\setminus J)|-2n+2\ge \delta|V(J)|-8|V(J)|+14\ge 2\delta^2 -19\delta +38,
$$
and thus $N> (\delta -5)^2$. This shows the claim for $k=4$.
\begin{case}\label{c3}
$k=5$
\end{case}
Construct the graph $G$ as follows (see Figure~\ref{p4_5}). Let 
\begin{eqnarray*}
P^1&=&p_1p_2p_3p_4,\\ 
P^2&=&r_1^1\ldots (r_1^{\delta-3}=r_2^1)r_2^2\ldots (r_2^{\delta-2}=r_3^1)r_3^2\ldots (r_3^{\delta-2}=r_4^1)r_4^2\ldots r_4^{\delta-3},\\
P^3&=&v_{1,1}^1\ldots (v_{1,1}^{\delta-4}=v_{1,2}^1)v_{1,2}^2\ldots (v_{1,2}^{\delta-3}=v_{1,3}^{1})\ldots v_{4,\delta-3}^{\delta-4}
\end{eqnarray*}
be paths,
let $Z_{1,1}^1,\ldots ,~(Z_{1,1}^{\delta-3}=Z_{1,2}^1),\ldots ,~Z_{4,\delta-2}^{\delta-3}$ 
be complete graphs on $\delta+1$ vertices each with 
$v_{i,j}^k,w_{i,j}^k,x_{i,j}^k,y_{i,j}^k,z_{i,j}^k\in V(Z_{i,j}^k)$ and $\{ w_{i,j}^k,x_{i,j}^k\} =\{ y_{i,j}^{k+1},z_{i,j}^{k+1}\} $
for  $1\le i\le 4$, $1\le j\le \delta-2$, and $1\le k\le \delta-3$. 
Add edges $p_ir_i^j$, $r_i^jv_{i,j}^k$, $p_1y_{1,1}^1,~p_1^1z_{1,1}^1,~r_1^1y_{1,1}^1,~r_1^1z_{1,1}^1$, 
$p_4w_{4,\delta-3}^{\delta-4}$, $p_4x_{4,\delta-3}^{\delta-4}$, $r_4^{\delta-3}w_{4,\delta-3}^{\delta-4}$, 
and $r_4^{\delta-3}x_{4,\delta-3}^{\delta-4}$.
\begin{figure}[h]\label{p4_5}
\begin{center}
\scalebox{0.5}{\includegraphics{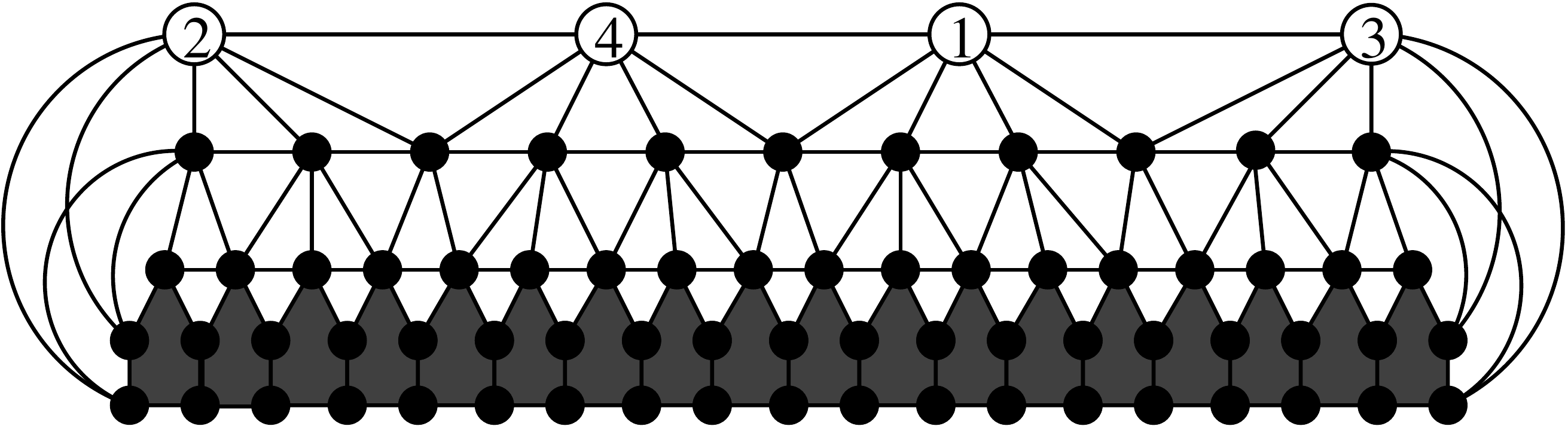}}
\end{center}
\caption{The graph $G$ in Case~\ref{c3}} 
\end{figure}

Then $\delta(G)=\delta$, $G$ is $5$-connected, and $N=|V(G)|=4\delta^3-33\delta^2+84\delta-58<4(\delta-2.7)^3$. 
Further, there is no path containing 
$p_3,~p_1,~p_4,~p_2$ in order. Therefore,  $\delta(5,P^4,N)>\sqrt[3]{\frac{N}{4}}+2.7$.

To show the upper bound for $\delta(5,P^4,N)$, assume that $G$ is a $5$-connected graph on $N$ vertices with minimum degree 
$\delta \ge 7$ (the statement is trivial for $\delta\le 6$),
and assume that $G$ contains vertices $a,~c,~b,~a'$, but no path contains the vertices in the given order.
Then $(G_{b,c},\{ b,b'\} ,\{ c,c'\} ,(a,a'))$ is an obstruction and has the structure described in Theorem~\ref{char}.
By the same argument as in Case~\ref{c2},
we can conclude that the first rung $R_1\subseteq L\subset G_{b,c}$ is of type (1) with 
$\{ b,b'\} =\{ x_1,y_1\}$ and $V(R_1)=\{ a,b,b',v_1\}$.

But since $d(b)\ge \delta$ and $G$ (and thus $G_{b,c}$) is $5$-connected, we have in fact that for $1\le i\le \delta-3$, 
$V(R_i)=\{ v_{i-1},v_i,b,b'\}$. Similarly, $V(R_i)=\{ v_{i-1},v_i,c,c'\}$ for $m-\delta+4\le i\le m$. Further, for $1\le i\le m$,
$R_i$ is either of type 
(1) with $v_{i-1}=v_i$ or $V(R_i)=\{ v_{i-1},x_{i-1},y_{i-1}, v_{i},x_{i},y_{i}\}$. Otherwise,
we could find a $4$-cut, or there would be a contradiction to Euler's formula in one of the planar subgraphs inside $R_i$.

Next, we will look at $N(a)\cap L$. If $av_i\in E(G_{b,c})$ for some $2\le i\le m-1$, then $\{ a,v_i,x_i,y_i\}$ is a $4$-cut
of $G_{b,c}$, so $N(a)\cap L\subseteq \{ v_1,v_m,b,b'\}$. Similarly,  $N(a')\cap L\subseteq \{ v_0,v_{m-1},c,c'\}$.
Therefore, 
$$
|(N(a)\cup N(a'))\setminus L|\ge 2(\delta-3)-1=2\delta-7.
$$
Observe that $V(J)\cup \{ b,c\}$ induce a planar graph.
The vertices in 
$$(N(a)\cup N(a'))\setminus L\cup\{ v_1,\ldots v_{\delta-3},v_{m-\delta+4},\ldots v_{m-1}\}$$ 
induce a subgraph of a path, as otherwise $G[V(J)\cup \{ b,c\}]$ would contain a separating
cycle of length at most $6$, which would lead to a contradiction with Euler's formula.
This implies that 
$$|V(J)|\ge (4\delta-10)(\delta-3)+4=4\delta^2-22\delta +34.$$
For every $1\le j\le m-1$, if $v_j\ne v_{j+1}$, then $V(L)\cap N(v_j)\cap N(v_{j+1})\subseteq \{ x_j,y_j\}$
due to the ladder structure of $L$. Every other vertex in $V(L)$ has at most one neighbor in $V(J)$.
Noting that $|V(J)|\ge m+2\delta-7$ and $|V(L\cap J)|=n+1$, we have (using~(\ref{e1}))
$$
|V(G)|\ge |E(J,G\setminus J)|-2(m-1)+|V(J)|\ge \delta|V(J)|-7|V(J)|+4\delta 
>4(\delta-6)^3,
$$
and thus $\delta(5,P^4,N)< \sqrt[3]{\frac{N}{4}}+6$. The last inequality also gives us
$|V(L)|>4(\delta-4.2)^3$ for $\delta\ge 50$, so $\delta(5,P^4,N)< \sqrt[3]{\frac{N}{4}}+4.2+o(1)$.

\begin{case}\label{c4}
$k= 6$ 
\end{case}
It follows from Theorem~\ref{klinked} that $\delta(6,P^4,N)\le 10$, but we can do a little better.
First, we will construct a $6$-connected graph $G$ with $\delta(G)=7$, which is not $P^4$-linked.
This is a graph very similar to a graph constructed by Yu in~\cite{Y}, although there he falsely claims that this
graph is $7$-connected.

Choose $n$ large enough to be able to 
construct a $3$-connected plane graph $(J,w_0,\ldots,w_n)$
along the lines of the construction in Case~\ref{c3}. Add a ladder $L$ along
$J\cap L=w_0w_1\ldots w_n$ as follows. Add vertices
$x$, $y$, and $x_i$, $y_i$ for $6\le i\le n-5$, and edges $xw_j,~yw_{n-j}$ for $0\le j\le 4$, 
$xx_6$, $xy_6$, $w_4x_6$, $w_4y_6$, $w_5x_6$, $w_5y_6$, $yx_{n-5}$, $yy_{n-5}$, $w_{n-4}x_{n-5}$,
$w_{n-4}y_{n-5}$, $w_{n-5}x_{n-5}$, $w_{n-5}y_{n-5}$, 
and all possible edges in $ \{ w_i,x_{i},y_{i},x_{i+1},y_{i+1}\}$ for $6\le i\le n-6$ (see Figure~\ref{p46}).

If we construct $J$ carefully, then $G$ is $6$-connected and has $\delta(G)=7$. 
But  there is no path through $p_1=w_0,~p_2=y,~p_3=x,~p_4=w_n$ in order. This construction works for $N= 394$, 
and with slight adjustments for all $N\ge 418$. Note that $\{ x_i,y_i,w_i,w_{i+1},x_{i+2},y_{i+2}\}$ 
is a $6$-cut for $6\le i\le n-7$, so $G$ is not $7$-connected. 
\begin{figure}[h]
\begin{center}
\scalebox{0.5}{\includegraphics{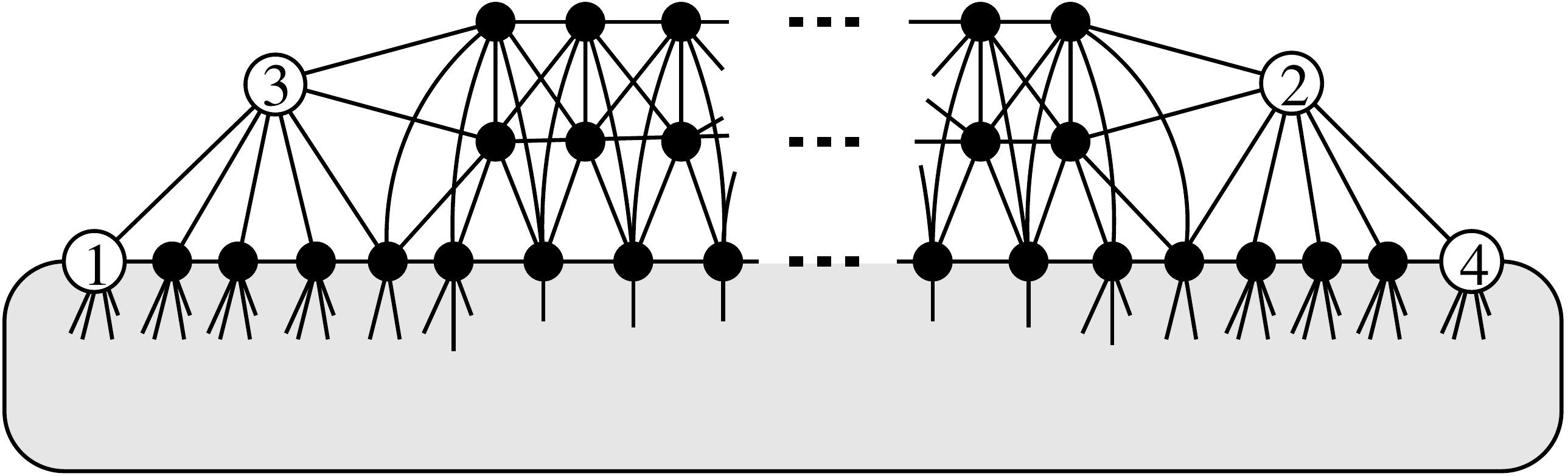}}
\end{center}
\caption{The graph $G$ in Case~\ref{c4}} \label{p46}
\end{figure}

To conclude that $\delta(6,P^4,N)\le 8$, assume that there is a $6$-connected graph $G$ with  $\delta(G)\ge 8$, which
is not $P^4$-linked, i.e., $(G_{b,c},\{ b,b'\}, \{ c,c'\} ,(a,a'))$ is an obstruction for some $a,a',b,c\in V(G)$, and
has the structure given by Theorem~\ref{char}.

Following the same arguments as in Case~\ref{c3}, we can see that $V(R_i)=\{ v_{i-1},x_{i-1},y_{i-1}, v_{i},x_{i},y_{i}\}$
for $1\le i\le m$, while $\{ x_i,y_i\} =\{ b,b'\}$ and $\{ x_{m-i},y_{m-i}\} =\{ c,c'\}$ for $0\le i \le 5$.
We now introduce two new types (8) and (8') of rungs $(R,(a,b,b'),(a',c,c'))$. Note that these rungs have proper
$3$-separations, as opposed to all other types.

\begin{tabular}{l p{0.8\textwidth} }
(8)& $|\{ a,b,b',a',c,c'\}|=6$, and $N(a)\subseteq \{ a',b,b'\}$;\\
(8')& $|\{ a,b,b',a',c,c'\}|=6$, and $N(a')\subseteq \{ a,c,c'\}$.
\end{tabular}

In fact, these types of rungs are each two rungs of type (1) in a row, so Theorem~\ref{char} is still valid
if we allow them. But we will use this notation in the following arguments.

Without loss of generality we may assume that $v_i\ne v_{i+1}$ for $1\le i\le m-1$. Otherwise, either 
$(G[V(R_{i}\cup R_{i+1})],(v_{i-1},x_{i-1},y_{i-1}),(v_{i+1},x_{i+1},y_{i+1}))$ is a rung (possibly of type (8) or (8')), 
$N(x_{i+1})\subseteq \{ v_i,x_i,y_i,y_{i+1},v_{i+2},x_{i+2},y_{i+2}\}$,
or $N(y_{i+1})\subseteq \{ v_i,x_i,y_i,x_{i+1},v_{i+2},x_{i+2},y_{i+2}\}$, contradicting $\delta(G)\ge 8$.
Similarly, we may assume that $v_0\ne v_1$.

Assume that $N(v_i)\cap V(J)\subseteq \{ v_{i-1},v{i+1},u\}$ for some $1\le i\le m-1$ and some $u\in V(J)$. Then 
$\{ x_{i},y_i\} \not\subset \{ x_{i-1},y_{i-1},x_{i+1},y_{i+1}\}$ since $d(v_i)\ge 8$. This implies that
$x_{i},y_i \notin \{ x_{i-1},y_{i-1},x_{i+1},y_{i+1}\}$, since $d(x_i),d(y_i)\ge 8$, and in fact 
$N(x_i)\setminus y_i=N(y_i)\setminus x_i=\{ v_{i-1},x_{i-1},y_{i-1},v_i,v_{i+1},x_{i+1},y_{i+1}\}$.
But since $d(v_i)\ge 8$, $N(v_i)\cap  \{ x_{i-1},y_{i-1},x_{i+1},y_{i+1}\}\ne \emptyset$, 
but this contradicts the fact that $L$ is a ladder. Therefore, every $v_i$ has at least $4$ neighbors in $V(J)$.
But this impossible by a simple application of Euler's formula.

\begin{case}\label{c5}
$k\ge 7$ 
\end{case}
We only need to show that every $7$-connected graph is $P^4$-linked, the other bounds follow from Case~\ref{c4}.
We show the slightly stronger statement that obstructions are at most $6$-connected.

Let $(G,(a,b,b'),(a',c,c'))$ be an obstruction, and suppose that $G$ is $7$-connected.
Following the same arguments as in Case~\ref{c3}, we can see that $V(R_i)=\{ v_{i-1},x_{i-1},y_{i-1}, v_{i},x_{i},y_{i}\}$
for $1\le i\le m$, while $\{ x_i,y_i\} =\{ b,b'\}$ and $\{ x_{m-i},y_{m-i}\} =\{ c,c'\}$ for $0\le i \le 3$.

Without loss of generality we may assume that $v_i\ne v_{i+1}$ for $3\le i\le m-4$. Otherwise, either 
$(G[V(R_{i}\cup R_{i+1})],(v_{i-1},x_{i-1},y_{i-1}),(v_{i+1},x_{i+1},y_{i+1}))$ is a rung  
(possibly of type (8) or (8')) or $\{ v_i,x_i,y_i,v_{i+2},x_{i+2},y_{i+2}\}$ is a cut set.

Assume that $|N_J(v_i)|,|N_J(v_{i+1})|\le 3$ for some $2\le i\le m-2$. 
We will consider 
$$S=\{ x_i,y_i,x_{i+1},y_{i+1}\}\setminus\{ x_{i-1},y_{i-1},x_{i+2},y_{i+2}\}.$$
To start with,
$S\ne\emptyset$, otherwise either $d(v_i)<7$, $d(v_{i+1})<7$,
or $L$ is not a ladder. 
If there is an edge from $v_{i-1}$ into $S$, then $v_ix_{i-1},v_iy_{i-1}\notin E(G)$, otherwise $R_i$ is not a rung.
As $d(v_i)\ge 7$, $|\{ x_i,y_i,x_{i+1},y_{i+1}\} |=4$ and $\{ x_i,y_i,x_{i+1},y_{i+1}\}\subset N(v_i)$. This implies that
there is no edge from $v_{i+1}$ to $\{ x_i,y_i\}$, otherwise $R_{i+1}$ is not a rung. But now,
$\{ v_{i-1},x_{i-1},y_{i-1},v_i,x_{i+1},y_{i+1}\}$ is a cut set, a contradiction. Thus, there is no edge from 
 $v_{i-1}$ into $S$. Similarly, there is no edge from $v_{i+2}$ into $S$. But this implies that
$\{ x_{i-1},y_{i-1},v_i,v_{i+1},x_{i+2},y_{i+2}\}$ is a cut set, a contradiction.
Therefore, at least one of  $|N_J(v_i)|,|N_J(v_{i+1})|$ must be greater than $3$ for $2\le i\le m-2$.

Now consider $J$ and $C=J\cap L$. Without loss of generality we may assume that $C$ is in fact a cycle, otherwise we may add
the missing edges, and the resulting graph is still an obstruction. Since $G$ is $7$-connected, $C$ has no chords,
and $J\setminus C$ is connected. Let $B$ be an end block of $J\setminus C$, and $x\in V(B)$ the only cut vertex of $J\setminus C$ 
in $B$ (if $B\ne J\setminus C$). 
$B$ inherits a plane embedding from the embedding of $J$, and all the vertices on the outer
face of this embedding (other than $x$) have degree $d_B(v)\ge 4$ in $B$ by the argument in the last paragraph
(and thus $|V(B)|\ge 5$ and $d_B(x)\ge 2$).
Suppose there are $k$ (including $x$) vertices on the outer face, and $\ell$ vertices not on the outer face.
For those internal vertices, we have $d_B(v)=d(v)\ge 7$. If we now connect all vertices on the outer face with 
an additional vertex $y$,
the resulting graph $B'$ is still planar. But
$$|E(B')|\ge \frac{4k+7\ell-2}{2}+k\ge 3(k+\ell+1)-4>
3|V(B')|-6,$$
contradicting the planarity of $B'$.   
\end{proof}

\section{$K^2\cup P^3$}

\begin{theorem}\label{Tk2p3}
Let $N\ge 29$. Then
$$
\begin{array}{rcccl}
&&\delta(k,K^2\cup P^3,N)&=&\left\lceil \tfrac{N+2}{2}\right\rceil ,\mbox{ if }k\le 3,\\
\smallskip
&&\delta(4,K^2\cup P^3,N)&=&\left\lceil \tfrac{N+1}{2}\right\rceil ,\\
\sqrt{\tfrac{N-1}{2}}+2.25&< & \delta(5,K^2\cup P^3,N) &< &\sqrt{3N}+4,\\
6~\le ~8-o(1)&\le&\delta(6,K^2\cup P^3,N)&\le& 10,\\ 
k&\le & \delta(k,K^2\cup P^3,N) &\le & \max\{ k,10\},\mbox{ if }k\ge 7.
\end{array}
$$
\end{theorem}
\begin{proof}
\begin{case}
$k\le 3$
\end{case}
By Fact~\ref{f3}, every $K^2\cup P^3$-linked graph is $4$-connected. This implies that 
$\delta(k,K^2\cup P^3,N)\ge\left\lceil \frac{N+2}{2}\right\rceil$. 
Equality follows from 
the next case as every graph with minimum degree $\left\lceil \frac{N+2}{2}\right\rceil$ is $4$-connected.
\begin{case}\label{c6}
$k=4$
\end{case} 
To show that
$\delta(4,K^2\cup P^3,N)> \lfloor\frac{N}{2}\rfloor$ consider a graph $G$
consisting of two complete graphs $G_1$ and $G_2$ with $|G_1|=\lceil\frac{N+2}{2}\rceil$, 
$|G_2|=\lfloor\frac{N+2}{2}\rfloor$,
and $|G_1\cap G_2|=2$, and two additional edges $p_2b,~p_1a$ with $p_2, a\in V(G_1\setminus G_2)$ 
and $p_1,b\in V(G_2\setminus G_1)$.
If we choose $p_3\in V(G_2)$, then $G$ contains no $(K^2\cup P^3)$-linkage consisting of a $a-b$ 
path and a $p_1-p_2-p_3$ path. 

Now let $G$ be a $4$-connected graph on $N$ vertices with minimum degree
$\delta(G)\ge \left\lceil \frac{N+1}{2}\right\rceil$. If $G$ is $5$-connected, then $G$ is $K^2\cup P^3$-linked
by the next case,
so we may assume that $G$ has a $4$-separation $(A,B)$.
If $N$ is even, then $|A|=|B|=\frac{N+4}{2}$, and $G[A]$ and $G[B]$ can have missing
edges only inside $A\cap B$.
Such a graph can easily be seen to be  $(K^2\cup P^3)$-linked.

If $N$ is odd,  then we may assume that $|A|=\frac{N+3}{2}$ and $|B|=\frac{N+5}{2}$.
Again, $G[A]$ is complete up to some missing edges inside $A\cap B$. Further, $G[B]$
can only miss a matching and then some edges inside $A\cap B$. In particular,
there exists a matching with four edges between $B\setminus A$ and $A$.
Such a graph can easily be seen to be  $(K^2\cup P^3)$-linked: 
given vertices $a,b,p_1,p_2,p_3\in V(G)$, find a shortest $p_1-p_2-p_3$ path 
using the fewest possible vertices in $A\cap B$ and none of $a,b$ such
that the remaining graph is still connected.
\begin{case}\label{c7}
$k=5$
\end{case}
For the lower bound, we will construct a graph similar to those in the proof of Theorem~\ref{thmp4}.
Let $\delta\ge 5$. Let $Z_i$, $1\le i\le 2\delta -7$  be complete graphs with $|V(Z_i)|=\delta+1$. Let
$\{ v_i,x_{i-1},y_{i-1},x_i,y_i\}\subset V(Z_i)$, 
and otherwise the $V(Z_i)$ are disjoint. 
Let $V(G)=\{ a,p_2\}\cup \bigcup V(Z_i)$, let $p_1=x_{2\delta-7},~p_3=y_{2\delta-7}$
and $b=x_0$. Add the edges $ap_i,~bp_i$ for $1\le i\le 2$, $p_2y_0$, $av_{2\delta -6-j},~p_2v_{j}$ for 
$1\le j\le \delta-3$, and $v_jv_{j+1}$ for $1\le j\le 2\delta-8$ (see Figure~\ref{k2p3}).

\begin{figure}[h]
\begin{center}
\scalebox{0.5}{\includegraphics{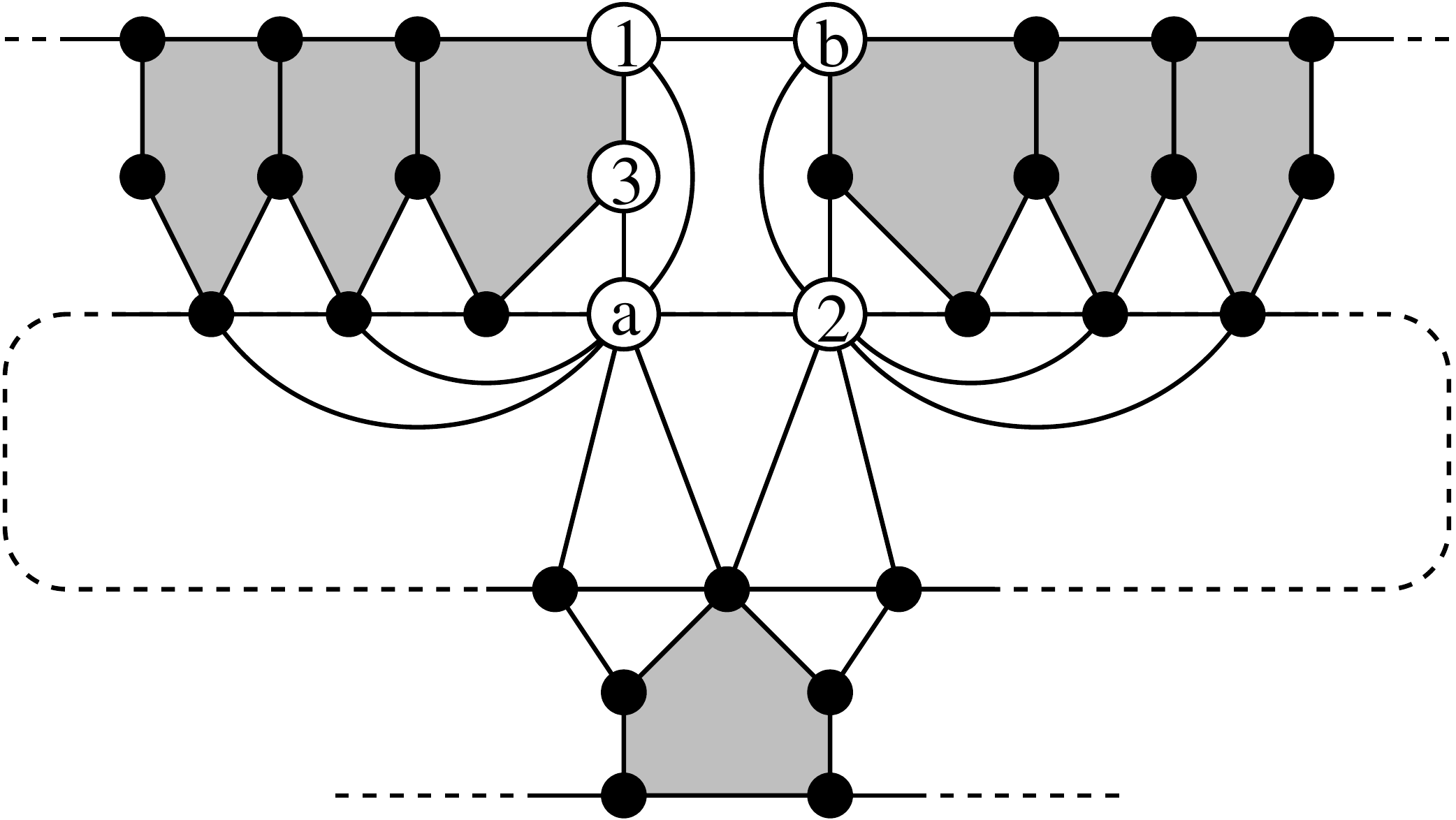}}
\end{center}
\caption{The graph $G$ in Case~\ref{c7}}\label{k2p3}
\end{figure}

Then $\delta(G)=\delta$, $G$ is $5$-connected, and $N=|V(G)|=2\delta^2-9\delta+11<2(\delta-2.25)^2+1$. 
Further, there is not an $a-b$ path and a $p_1-p_2-p_3$ path, which are disjoint. 
Therefore,  $\delta(5,K^2\cup P^3,N)>\sqrt{\frac{N-1}{2}}+2.25$.

For the upper bound, we will first show the following claim.
\begin{claim}
Let $G$ be a graph with minimum degree $\delta=\delta(G)\ge 8$, 
let $X=\{ a,b,p_1,p_2,p_3\}\subset V(G)$. 
Suppose that $G$ has no $4$-separation $(A,B)$ with
$X\subseteq A$. Suppose that  $G$ does not contain disjoint
$a-b$ and $p_1-p_2-p_3$ paths, 
and suppose that no edge can be added without destroying this property.

Then for every $5$-separation $(A,B)$ with
$X\subseteq A$ and $p_2\notin  B$, $A\cap B$ induces a $K^5$.
\end{claim}
For the sake of contradiction, choose a $5$-separation $(A,B)$ with
$X\subseteq A$ and $p_2\notin B$, for which $G[A\cap B]$ is not complete, such that $B$ is minimal.
But now it is easy to see that $(B,A\cap B)$ is linked (you may apply Theorem~\ref{3linked} to $G[B]+x$,
where the added vertex $x$ is joined to every vertex, concluding that $(G[B]+x,(A\cap B)\cup\{ x\})$ is linked).
This shows that adding edges within $A\cap B$
will not create disjoint $a-b$ and $p_1-p_2-p_3$ paths, showing the claim.

Next, let $N\ge 29$, $\delta=\delta(G)\ge \sqrt{3N}+5$ and consider the set $\cS$
of all $5$-separations  $(A,B)$ with
$X\subseteq A$, $p_2\notin B$, and $B$ is maximal (i.e., there is no such separation $(A',B')$
with $B\subsetneq B'$). For every such separation, $|B|\ge \delta+1$, and for every pair of two 
such separations, $B\cap B'\subset A\cap A'$.

Let
$$\cA :=V(G)\setminus\bigcup_{(A,B)\in \cS}B ~,~\cB:=\bigcup_{(A,B)\in \cS}A\cap B.$$
Note that $|\cB|\le 5|\cS|\le 5\frac{N-|\cA|- |\cB|}{\delta -4}$.
Consider the graph $G'$ obtained from $G[\cA\cup \cB]$ by adding a
vertex $p_2'$ with $N(p_2')=N[p_2]$. Note that $G'$ has no
$5$-separation $(A,B)$ with $X\cup p_2'\subseteq A$.

If $|\cA|\ge \delta$, then (using that every vertex in $\cB$ has at least $6$ neighbors in $G'$,
$|\cB|\le 5|\cS|\le 5\frac{N}{\delta-4}<\frac{5}{\sqrt{3}}\sqrt{N}$
and $\delta\ge \sqrt{3N}+4$)
$$
|E(G')|\ge \tfrac{\delta}{2}(|\cA|+1)+3|\cB|\ge  5|\cA|+5|\cB|-9=5|V(G')|-14.
$$
If $6\le |\cA|\le \delta$, then 
\begin{multline*}
|E(G')|\ge (\delta-\tfrac{1}{2}|\cA|)(|\cA|+1)+2|\cB|\\
\ge (|\cA|-4)\sqrt{3N}-\tfrac{|\cA|}{2}(|\cA|+1)+ 5|\cA|+5|\cB|-9\\
\ge 5|\cA|+5|\cB|-9=5|V(G')|-14.
\end{multline*}
Therefore, $(G',X\cup p_2')$ is linked by Theorem~\ref{3linked}, and we can find
the desired linkage in $G$, a contradiction. Thus,  $|\cA|\le 5$.

Finally, if $|\cA|\le 5$, note that if $p_2$ has more than $3$
neighbors in some $B$, then $G'$ contains a $K^6$ and so $(G',X\cup p_2')$ is
linked, a contradiction.
Thus, 
$$
\delta-4\le |N(p_2)\cap \cB|\le 3|\cS|\le 3\frac{N-|\cA\cup\cB|}{\delta-4},
$$
contradicting the fact that $\delta\ge \sqrt{3N}+4$.
This shows that $\delta(5,K^2\cup P^3,N)<\sqrt{3N}+4$ for $N\ge 29$.

\begin{case}
$k\ge 6$
\end{case}
The lower bound for $\delta(k,K^2\cup P^3,N)$ follows from Theorem~\ref{thmp4}, 
the upper bound follows from Theorem~\ref{3linked}.

\end{proof}

\section{$K^2\cup C^2$ 
and $P^3\cup P^3$}
By Fact~\ref{f2}, every $(K^2\cup P^3)$-linked graph is $(K^2\cup C^2)$-linked. Thus, all the upper bounds
in Theorem~\ref{Tk2p3} apply to $\delta(k,K^2\cup C^2,N)$ as well. As for lower bounds, note that all the examples
in the proof of Theorem~\ref{Tk2p3} with $k\le 5$ yield the same lower bounds for 
$\delta(k,K^2\cup C^2,N)$ (none of them contains a disjoint $a-b$ path and a cycle through $p_1$ and $p_2$). For $k=6$,
we can employ again the example in Case~\ref{c4} in the proof of Theorem~\ref{thmp4}, and note that
this graph does not contain  a disjoint $p_1-p_2$ path and a cycle through $p_3$ and $p_4$. Therefore, we have the following theorem.

\begin{theorem}\label{Tk2c2}
Let $N\ge 29$. Then
$$
\begin{array}{rcccll}
&&\delta(k,K^2\cup C^2,N)&=&\left\lceil \tfrac{N+2}{2}\right\rceil ,&\mbox{ if }k\le 3,\\
\smallskip
&&\delta(4,K^2\cup C^2,N)&=&\left\lceil \tfrac{N+1}{2}\right\rceil ,&\\
\sqrt{\tfrac{N-1}{2}}+2.25&< & \delta(5,K^2\cup C^2,N) &< &\sqrt{3N}+4,&\\
6~\le ~8-o(1)&\le&\delta(6,K^2\cup C^2,N)&\le& 10,&\\ 
k&\le & \delta(k,K^2\cup C^2,N) &\le & \max\{ k,10\},&\mbox{ if }k\ge 7.
\end{array}
$$
\end{theorem}

Now that we have considered all multigraphs with up to 3 edges, let us consider graphs
$H$ with 4 edges. We can prove the following theorem about $H=P^3\cup P^3$.

\begin{theorem}\label{Tp3p3}
Let $N$ be large enough.
Then
$$
\begin{array}{rcccll}
&&\delta(6,P^3\cup P^3,N)&=&\left\lceil \tfrac{N+2}{2}\right\rceil ,&\\
\sqrt{\tfrac{N-2}{2}}+3.25&< & \delta(7,P^3\cup P^3,N) &< &\sqrt{5N}+6,&\\
8&\le&\delta(8,P^3\cup P^3,N)&\le& 40&
\end{array}
$$
\end{theorem}

\begin{proof}
For the upper bounds, we use Theorem~\ref{Tlk2p3} and that 
$
\delta(k,P^3\cup P^3,N)\le \delta(k,2K^2\cup P^3,N)$ by Fact~\ref{f2}. 
For the lower bounds we find examples.
\begin{case}
$k=6$
\end{case}
Let $G$ consist of two complete graphs $G_1$ and $G_2$ with $|G_1|=\lceil\frac{N+3}{2}\rceil$, 
$|G_2|=\lfloor\frac{N+3}{2}\rfloor$,
and $|G_1\cap G_2|=3$, and three additional edges $p_1q_1,~p_2q_2,~p_3q_3$ with $p_i\in V(G_1\setminus G_2)$ 
and $q_i\in V(G_2\setminus G_1)$.
Then $G$ contains no $(P^3\cup P^3)$-linkage consisting of a $p_1-q_2-p_3$ 
path and a $q_1-p_2-q_3$ path.
\begin{case}\label{4.2.2}
$k=7$
\end{case}
Let $\delta\ge 7$. Let $Z_i$, $1\le i\le 2\delta -9$,  be complete graphs with $|V(Z_i)|=\delta+1$.\\ 
Let
$\{ v_i, x_{i-1}, y_{i-1} , z_{i-1}, x_i, y_i, z_i\}\subset V(Z_i)$, 
and otherwise the $V(Z_i)$ are disjoint. 
Let $V(G)=\{ p_2,q_2\}\cup \bigcup V(Z_i)$, let $p_1=x_{2\delta-9},~p_3=y_{2\delta-9}$, $q_1=x_0$ and $q_3=y_0$. Add the edges $p_iq_i$ for $1\le i\le 3$, $p_1p_2$, $p_2p_3$, $p_2z_{2\delta-9}$, $q_1q_2$, $q_2q_3$, $q_2z_0$, $q_2v_{j},~p_2v_{2\delta -8-j}$ for 
$1\le j\le \delta-4$, and $v_jv_{j+1}$ for $1\le j\le 2\delta-10$ (see Figure~\ref{p3p3}).

\begin{figure}[h]
\begin{center}
\scalebox{0.5}{\includegraphics{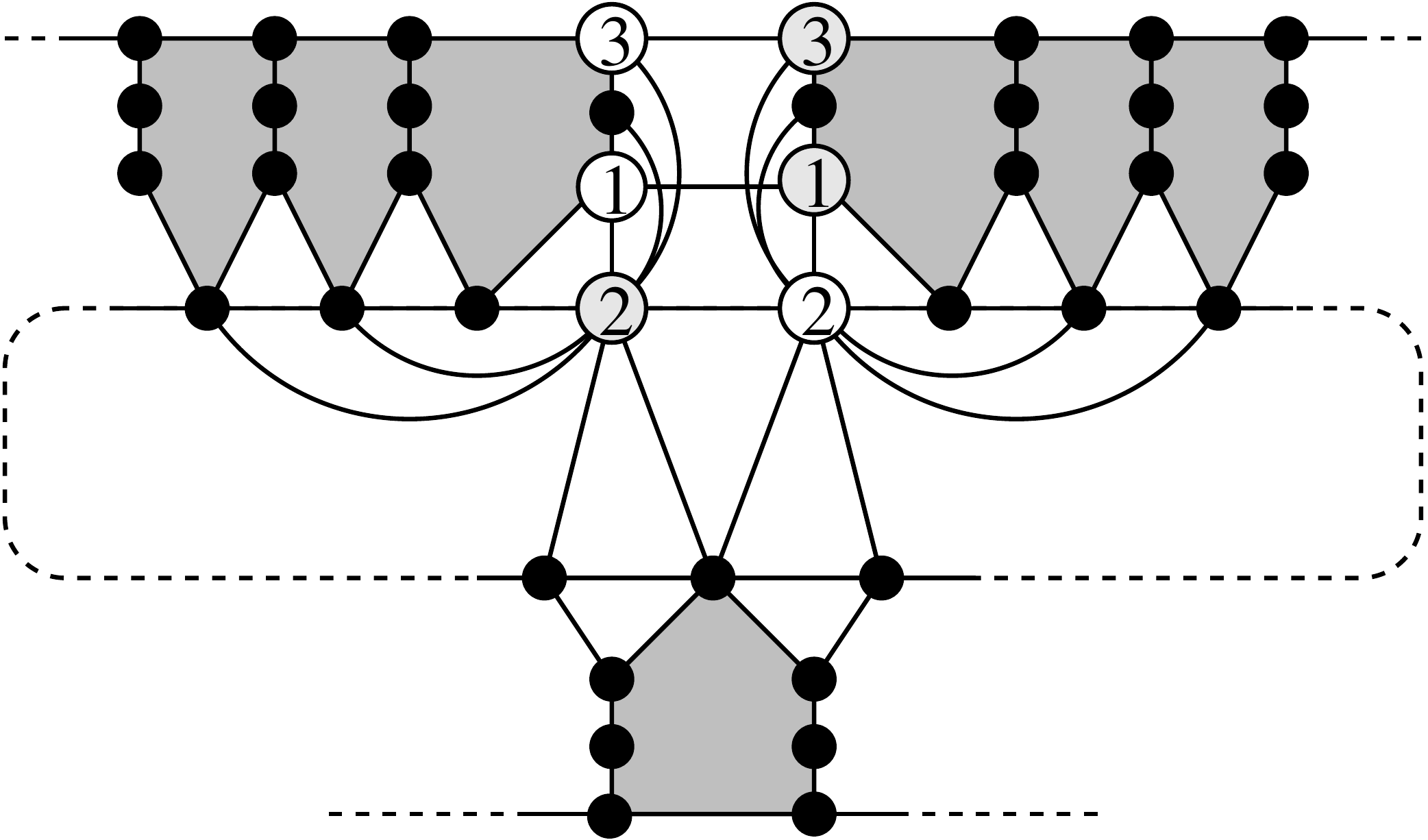}}
\end{center}
\caption{The graph $G$ in Case~\ref{4.2.2}}\label{p3p3}
\end{figure}

Then $\delta(G)=\delta$, $G$ is $7$-connected, and $N=|V(G)|=
2\delta^2-13\delta+23<2(\delta-3.25)^2+2$. 
Further, there is not an $p_1-q_2-p_3$ path and a $q_1-p_2-q_3$ path, which are disjoint. 
Therefore,  $\delta(7,P^3\cup P^3,N)>\sqrt{\frac{N-2}{2}}+3.25$.

\end{proof}

\section{Bipartite $H$ with small components}

Very similarly to Theorems~\ref{Tk2p3} and~\ref{Tk2c2}, 
we obtain the following result.
\begin{theorem}\label{Tlk2p3}
Let $N$ be large enough, $\ell\ge 1$ and $H\in \{\ell~K^2\cup P^3,\ell~K^2\cup C^2\}$. Then
$$
{\small
\begin{array}{rcccl}
&&\delta(2\ell+1,H,N)&=&\left\lceil \frac{N+\ell+1}{2}\right\rceil ,\\
&&\delta(2\ell+2,H,N)&=&\left\lceil \frac{N+\ell}{2}\right\rceil ,\\
\sqrt{\frac{N-2\ell+1}{2}}+2\ell+0.25&< & \delta(2\ell+3,H,N)&<
&\sqrt{(2\ell+1)N}+2\ell+2,\\
2\ell+4&\le & \delta(2\ell+4,H,N)&\le & 10(\ell+2).
\end{array}
}
$$
\end{theorem}
\begin{proof}
The proof follows arguments very similar to the proofs of Theorems~\ref{Tk2p3} and~\ref{Tk2c2}, and is left to the reader. The only 
inequality we elaborate on here is the lower bound for $k=2\ell+3$. 
For this, add $2\ell-2$ vertices to the bounding graph in Theorem~\ref{Tk2c2}, and connect them with all other vertices. 
Making these new vertices the terminals of the extra $K^2$s it is easy to see that this graph is not $(\ell~K^2\cup C^2)$-linked.
\end{proof}

\section{Conclusion and open questions}
We have determined $\delta(k,H,N)$ for all $H$ with up to three edges, up to some small constant factors. In every case, $\delta(k,H,N)=\Theta(N^{1/\ell})$. Is this the case for all $k$ and $H$?

We know $\delta(k,H,N)$ only for few $H$ with more than three edges. Very interesting should be the cases $H=C^4$ (as almost always), 
$H=K^2\cup K_{1,3}$ and $H=K^2\cup P^4$. 
In the last case, we know for sufficiently large $N$ (with a proof similar to Theorem~\ref{Tp3p3}) that 
$$
\begin{array}{rcccll}
&&\delta(6,K^2\cup P^4,N)&=&\left\lceil \tfrac{N+2}{2}\right\rceil ,&\\
\sqrt[3]{N-2}+4.7&\le & \delta(7,K^2\cup P^4,N) &< &\sqrt{5N}+6,&\\
8&\le&\delta(8,K^2\cup P^4,N)&\le& 40,&
\end{array}
$$
but this leaves quite a gap between the bounds for $\delta(7,K^2\cup P^4,N)$.

\bibliographystyle{amsplain}

\end{document}